\documentclass{amsart}
\usepackage{amssymb, enumerate, palatino, amsfonts, latexsym}
\usepackage[mathscr]{euscript}
\usepackage{xy} \xyoption{all}
\usepackage{aliascnt}
\usepackage{hyperref}

\DeclareMathOperator{\id}{id}
\DeclareMathOperator{\gen}{gen}
\DeclareMathOperator{\Prim}{Prim}
\newcommand{\mx}{\text{max}}

\newcommand{\setDefSep}[0]{\: |\:}
\newcommand{\e}{\varepsilon}
\newcommand{\KK}{{\mathbb{K}}}
\newcommand{\RR}{{\mathbb{R}}}
\newcommand{\NN}{{\mathbb{N}}}
\newcommand{\CC}{{\mathbb{C}}}
\newcommand{\Cs}{{$C^*$-al\-ge\-bra}}
\newcommand{\termDef}[1]{\textbf{#1}\index{#1}}

\newtheorem{lma}{Lemma}[section]

\newaliascnt{thmCt}{lma}
\newtheorem{thm}[thmCt]{Theorem}
\aliascntresetthe{thmCt}

\newaliascnt{corCt}{lma}
\newtheorem{cor}[corCt]{Corollary}
\aliascntresetthe{corCt}

\newaliascnt{propCt}{lma}
\newtheorem{prop}[propCt]{Proposition}
\aliascntresetthe{propCt}

\newaliascnt{defnCt}{lma}

\aliascntresetthe{defnCt}

\newtheorem*{thm*}{Theorem}
\newtheorem*{cor*}{Corollary}
\newtheorem*{prop*}{Proposition}

\theoremstyle{definition}
\newtheorem{pargr}[lma]{}

\newaliascnt{rmkCt}{lma}
\newtheorem{rmk}[rmkCt]{Remark}
\aliascntresetthe{rmkCt}

\newaliascnt{questCt}{lma}
\newtheorem{quest}[questCt]{Question}
\aliascntresetthe{questCt}

\newaliascnt{exmplCt}{lma}
\newtheorem{exmpl}[exmplCt]{Example}
\aliascntresetthe{exmplCt}

\begin{document}

\title[The generator problem for $\mathcal{Z}$-stable $C^{*}$-algebras]{The generator problem for $\mathcal{Z}$-stable $C^*$-algebras}%

\author{Hannes Thiel}
\address{Department of Mathematical Sciences, University of Copenhagen, Universitetsparken 5, DK-2100, Copenhagen \O, Denmark}
\email{thiel@math.ku.dk}

\author{Wilhelm Winter}
\address{Mathematisches Institut der Universit\"at M\"unster, Einsteinstr. 62, 48149 M\"unster, Germany}
\email{wwinter@uni-muenster.de}

\thanks{
This research was partially supported by the Centre de Recerca Matem\`{a}tica, Barcelona.
The first named author was partially supported by the Danish National Research Foundation through the Centre for Symmetry and Deformation, Copenhagen.
The second named author was partially supported by EPSRC Grants  EP/G014019/1  and EP/I019227/1.
}

\subjclass[2000]%
{Primary
46L05, 
46L85; 
Secondary
46L35. 
}

\keywords{$C^*$-algebras, generator problem, single generation, $\mathcal{Z}$-stability}

\date{\today}

\begin{abstract}
    The generator problem was posed by Kadison in 1967, and it remains open until today.
    We provide a solution for the class of \Cs{s} absorbing the Jiang-Su algebra $\mathcal{Z}$ tensorially.
    More precisely, we show that every unital, separable, $\mathcal{Z}$-stable \Cs{} $A$ is singly generated, which means that there exists an element $x\in A$ that is not contained in any proper sub-\Cs{} of $A$.

    To give applications of our result, we observe that $\mathcal{Z}$ can be embedded into the reduced group \Cs{} of a discrete group that contains a non-cyclic, free subgroup.
    It follows that certain tensor products with reduced group \Cs{s} are singly generated.
    In particular, $C^*_r(F_\infty)\otimes C^*_r(F_\infty)$ is singly generated.
\end{abstract}

\maketitle

\section{Introduction}

    By an operator algebra we mean a $^{*}$-subalgebra of $B(H)$ that is either closed in the norm topology (a concrete \Cs{}) or the weak operator topology (a von Neumann algebra).
    One way of realizing an operator algebra is to take a subset of $B(H)$ and consider the smallest operator algebra containing it.

    In a trivial way, every operator algebra can be obtained this way.
    The situation becomes interesting if one imposes restrictions on the generating set, and one natural possibility is to require that it consists of only one element, i.e., to consider operator algebras that are generated by a single operator.
    It is an old problem to determine which operator algebras arise this way.

    More generally, one tries to compute the minimal number of elements that generate a given operator algebra, see \ref{pargr:Generators}.
    It is often convenient to consider self-adjoint generators.
    Note that two self-adjoint elements $a,b$ generate the same operator algebra as the element $a+ib$.
    Thus, if we ask whether an operator algebra is singly generated, it is equivalent to ask whether it is generated by two self-adjoint elements.

    In the case of von Neumann algebras, the generator problem was included in Kadison's famous `Problems on von Neumann algebras', \cite{Kad1967}.
    This problem list has turned out to be very influential, yet its original form remains unpublished.
    It is indirectly available in an article by Ge, \cite{Ge2003}, where a brief summary of the developments around Kadison's famous problems is given.

\begin{quest}[{Kadison, \cite[Problem 14]{Kad1967}, see also \cite{Ge2003}}]
\label{quest:Gen_Problem_vNalg}
        Is every
separably-acting\footnote{A von Neumann algebra is called `separably-acting', or just `separable', if it is a subalgebra of $B(l^2\NN)$, or equivalently if it has a separable predual.}
        von Neumann algebra singly generated?
\end{quest}

    As noted in \cite{She2009}, there exist singly generated von Neumann algebras that are not sepa\-rably-acting.
    However, the separably-acting von Neumann algebras are the natural class for which one might expect single generation.
    The answer to \autoref{quest:Gen_Problem_vNalg} is still open in general, but many authors have contributed to show that large classes of separably-acting von Neumann algebras are singly generated.

    We just mention an incomplete list of results.
    It starts with von Neumann, \cite{vNe1931}, who showed that the abelian operator algebras named after him are generated by a single self-adjoint element, thus implicitly raising the generator problem.
    Some thirty years later, this was extended by Pearcy, \cite{Pea1962}, who showed that all von Neumann algebras of type $\text{I}$ are singly generated.
    Then Wogen, \cite[Theorem 2]{Wog1969}, proved that all properly infinite von Neumann algebras are singly generated, thus reducing the generator problem to the type $\text{II}_1$ case.

    Later, this was further reduced to the case of a $\text{II}_1$-factor by Willig, \cite{Wil1974}, and then to the case of a finitely-generated $\text{II}_1$-factor by Sherman, \cite[Theorem 3.8]{She2009}.
    This means that \autoref{quest:Gen_Problem_vNalg} has a positive answer if every separably-acting, finitely generated $\text{II}_1$-factor is singly generated.

    There are many properties known to imply that a $\text{II}_1$-factor is singly generated.
    We just mention that Ge and Popa, \cite[Theorem 6.2]{GePop1998}, show that every tensorially
non-prime\footnote{A $\text{II}_1$-factor $M$ is called tensorially non-prime if it is isomorphic to a tensor product, $M_1\bar{\otimes}M_2$, of two $\text{II}_1$-factors $M_1, M_2$.}
    $\text{II}_1$-factor is singly generated.
    Our main result \autoref{thm:Tensor_gen2} can be considered as a partial $C^{*}$-algebraic analog of this result.

    Let us also mention that the free group factors $W^{*}(F_k)$ are the outstanding examples of separably-acting von Neumann algebra for which it is not known whether they are singly generated.
    \\

    In the case of \Cs{s}, the generator problem is more subtle.
    There is already no obvious class of \Cs{s} for which one conjectures that they are singly generated.
    Every singly generated \Cs{} is separable\footnote{A \Cs{} is called `separable' if it contains a countable, norm-dense subset}.
    However, the converse is false, and counterexamples can be found among the commutative \Cs{s}.

    In fact, the \Cs{} $C_{0}(X)$ is generated by $n$ self-adjoint elements if and only if $X$ can be embedded into $\RR^n$.
    Thus, $C_{0}(X)$ is singly generated if and only if $X$ is planar, i.e., can be embedded into the plane $\RR^2$.

    It is easy to see that a \Cs{} $A$ is generated by $n$ self-adjoint elements if and only if its minimal unitization $\widetilde{A}$ is generated by $n$ self-adjoint elements.
    Therefore, we will mostly consider the generator problem for separable, unital \Cs{}.
    In that case, taking the tensor product with a matrix algebra has the effect of reducing the necessary number of generators.
    If $A$ is generated by $n^2+1$ self-adjoint elements, then $A\otimes M_n$ is singly generated, see e.g. \cite[Theorem 3]{Nag2004}.

    One derives the principle that a \Cs{} needs less generators if it is `more non-commutative'.
    Consequently, one might expect a (separable) \Cs{} to be singly generated if it is `maximally non-commutative'.
    As a non-unital instance of this principle, we note that the stabilization, $A\otimes\KK$, of a separable unital \Cs{} $A$ is singly generated, \cite[Theorem 8]{OlsZam1976}.
    In the unital case, there are at least three natural cases when one considers a \Cs{} $A$ to be `maximally non-commutative', which are the following:
    \begin{enumerate}[\quad (1)  ]
        \item
        $A$ contains a simple, unital, nonelementary sub-\Cs{},
        \item
        $A$ contains a sequence of pairwise orthogonal, full elements,
        \item
        $A$ has no finite-dimensional irreducible representations.
    \end{enumerate}

    In general, the implications $(1)\Rightarrow(2)\Rightarrow(3)$ hold; it is not known if the converses are true.

    Conditions $(2)$ and $(3)$ can also be considered for possibly non-unital \Cs{s},
    and we let $(2^*)$ be the weaker statement that $A$ contains \emph{two} orthogonal, full elements.
    The implication `$(3)\Rightarrow(2)$' holds exactly if the implication `$(3)\Rightarrow(2^*)$' holds.

    The Global Glimm halving problem asks the following:
    Given a (possibly non-unital) \Cs{} $A$ that satisfies condition $(3)$, does there exist a full map from the cone over $M_2$ to $A$?
    It is not known whether the Global Glimm halving problem has a positive answer,
    but if it does then it shows that implication `$(3)\Rightarrow(2)$' holds, since the cone over $M_2$ contains two orthogonal, full elements.

    Let us remark that the analogs of conditions $(1)-(3)$ for von Neumann algebras are all equivalent.
    In fact, if a von Neumann algebra $M$ has no finite-dimensional representations, then the hyperfinite $\text{II}_1$-factor $\mathcal{R}$ unitally embeds into $M$.

    Historically, the generator problem for \Cs{s} is mostly asked for \Cs{s} that are simple ore more generally have no finite-dimensional representations:

\begin{quest}
\label{quest:Gen_Problem_Calg_simple}
        Is every simple, separable, unital \Cs{} singly generated?
\end{quest}

\begin{quest}
\label{quest:Gen_Problem_Calg_no_fd_repr}
        Is a separable, unital \Cs{} singly generated provided it has no finite-dimensional irreducible representations?
\end{quest}

    The answers to both questions are open.
    A positive answer to \autoref{quest:Gen_Problem_Calg_no_fd_repr} implies a positive answer to \autoref{quest:Gen_Problem_Calg_simple}, of course.
    The converse is not clear.

    Let us mention some results that solve the generator problem for particular classes of separable \Cs{s}.
    It was shown by Topping, \cite{Top1968}, that  every UHF-algebra is singly generated.
    This was generalized by Olsen and Zame, \cite[Theorem 9]{OlsZam1976}, who showed that the tensor product, $A\otimes B$, of any separable, unital \Cs{} $A$ with a UHF-algebra $B$ is singly generated.

    Later, it was shown by Li and Shen, \cite[Theorem 3.1]{LiShe2008}, that every unital, approximately
divisible\footnote{A unital \Cs{} $A$ is `approximately divisible' if for every $\e>0$ and finite subset $F\subset A$ there exists a finite-dimensional, unital sub-\Cs{} $B\subset A$ such that $B$ has no characters and $\|xb-bx\|\leq\e\|b\|$ for all $x\in F, b\in B$.}
    \Cs{} is singly generated.
    This generalizes the result of Olsen and Zame, since the tensor product with a UHF-algebra is always approximately divisible.

    In this article we  prove that every separable, unital, $\mathcal{Z}$-stable \Cs{} is singly generated, see \autoref{prop:Gen2_Z-stable}.
    This generalizes the result of Li and Shen, since every approximately divisible \Cs{} is $\mathcal{Z}$-stable, see \cite[Theorem 2.3]{TomWin2008}.
    The notion of $\mathcal{Z}$-stability has proven to be very important in the classification program of nuclear \Cs{s}, see e.g.\ \cite{Win2007arx} or \cite{EllTom2008}, and it is has been shown that many nuclear, simple \Cs{s} are $\mathcal{Z}$-stable, see e.g.\ \cite{Win2010}. $\mathcal{Z}$-stability is also relevant in the non-nuclear context; for example, unital $\mathcal{Z}$-stable \Cs{s} satisfy Kadison's similarity property, see \cite{JohWin2011}. \\

    This paper proceeds as follows:
    \\

    In \autoref{sect:prelim} we set up our notation and give some basic facts about the generator rank, see \ref{pargr:Generators}, and $C_{0}(X)$-algebras, see \ref{pargr:C_X_alg}.
    \\

    \autoref{sect:results} contains the proof of our main result, which states that the tensor product $A\otimes_\mx B$ of two separable, unital \Cs{s} is singly generated, if $A$ satisfies condition $(2)$ from above (e.g. $A$ is simple and non-elementary) and $B$ admits a unital embedding of the Jiang-Su algebra $\mathcal{Z}$, see \autoref{thm:Tensor_gen2}.

    We derive that every separable, unital, $\mathcal{Z}$-stable \Cs{} is singly generated, see \autoref{prop:Gen2_Z-stable}.
    Our main result can be considered as a (partial) $C^*$-algebraic analog of a theorem of Ge and Popa, \cite[Theorem 6.2]{GePop1998}, which shows that a tensor product, $M\bar{\otimes} N$, of two $\text{II}_1$-factors $M,N$ is singly generated.
    In fact, we can reprove their theorem with our methods, see \autoref{prop:GePope}.
    \\

    In \autoref{sect:appl} we give further applications of our main theorem to tensor products with reduced group \Cs{s}.
    We first observe that $\mathcal{Z}$ embeds unitally into $C^*_r(F_\infty)$, the reduced group \Cs{} of the free group on infinitely many generators, see \autoref{pargr:Z-embedding_free_gp_alg}.
    Consequently, if a discrete group $\Gamma$ contains a non-cyclic free subgroup, then $\mathcal{Z}$ embeds unitally into $C^*_r(\Gamma)$, see \autoref{prop:Z-embedding_gps}.

    We deduce that tensor products of the form $A\otimes_\mx C^*_r(\Gamma)$ are singly generated if $A$ is a separable, unital \Cs{} satisfying condition $(2)$ from above, and $\Gamma$ is a group containing a non-cyclic free subgroup, see \autoref{prop:Gen2_tensor_free_gp_alg}.
    For example, $C^*_r(F_\infty)\otimes C^*_r(F_\infty)$ is singly generated, although this \Cs{} is not $\mathcal{Z}$-stable, see \autoref{pargr:free_gp_tensor_free_gp}.

\section{Preliminaries}
\label{sect:prelim}

\noindent
    By a morphism between \Cs{s} we mean a $^{*}$-homomorphism, and by an ideal of a \Cs{} we understand a closed, two-sided ideal.
    If $A$ is a \Cs{}, then we denote by $\widetilde{A}$ its minimal unitization.
    Often, we write $M_k$ for the \Cs{} of $k$-by-$k$ matrices $M_k(\CC)$.

\begin{pargr}
\label{pargr:Generators}
    Let $A$ be a \Cs{}, and $A_\text{sa}\subset A$ the subset of self-adjoint elements.
    We say that a set $S\subset A_\text{sa}$ generates $A$, denoted $A=C^\ast(S)$, if the smallest sub-\Cs{} of $A$ containing $S$ is $A$ itself.
    We denote by $\gen(A)$  the smallest number $n\in\{1,2,3,\ldots,\infty\}$ such that $A$ contains a generating subset $S\subset A_\text{sa}$ of cardinality $n$, and we call $\gen(A)$ the \termDef{generating rank} of $A$.

    We stress that for the definition of $\gen(A)$, the generators are assumed to be self-adjoint.
    Two self-adjoint elements $a,b$ generate the same \Cs{} as the (non-self-adjoint) element $a+ib$.
    Therefore, a \Cs{} $A$ is said to be singly generated if $\gen(A)\leq 2$.

    For more details on the generating rank we refer the reader to Nagisa, \cite{Nag2004}, where also the following simple facts are noted for \Cs{s} $A$ and $B$:
    \begin{enumerate}[(1)   ]
        \item
        $\gen(\widetilde{A})=\gen(A)$,
        \item
        $\gen(C^*(A,B))\leq\gen(A)+\gen(B)$, if $A,B$ are sub-\Cs{s} of a common \Cs{}, and where $C^*(A,B)$ denotes the sub-\Cs{} they generate together,
        \item
        $\gen(A\oplus B)=\max\{\gen(A),\gen(B)\}$ if at least one of the algebras is unital.
    \end{enumerate}

    Let $I\lhd A$ be an ideal in a \Cs{} $A$.
    It is easy to see that the generating rank of the quotient $A/I$ is not bigger than the generating rank of $A$, i.e., $\gen(A/I)\leq\gen(A)$, and the generating rank of $A$ can be estimated as $\gen(A)\leq\gen(I)+\gen(A/I)$.
    The following result gives an estimate for $\gen(I)$, and it is probably well-known to  experts;
    since we could not locate it in the literature,  we include a short proof.
\end{pargr}

\begin{prop}
\label{prop:Gen_ideal}
    Let $A$ be a \Cs{}, and $I\lhd A$ an ideal.
    Then $\gen(I)\leq\gen(A)+1$.
\end{prop}
\begin{proof}
    We may assume $\gen(A)$ is finite.
    So let $a_1,\ldots,a_k$ be a set of self-adjoint generators for $A$.
    Then $A$ and $I$ are separable, and so $I$ contains a strictly positive element $h$.
    It follows that $C^*(h)$ contains a quasi-central approximate unit, see \cite[Corollary 3.3]{AkePed1977} and \cite{Arv1977}.
    It is straightforward to show that $I$ is generated by the $k+1$ elements $h,ha_1h,\ldots,ha_kh$.
\end{proof}

\noindent
    The following result is attributed to Kirchberg in \cite{Nag2004}.

\begin{thm}[Kirchberg]
\label{prop:Gen2_prop_inf}
    Every separable, unital, properly infinite \Cs{} is singly generated.
\end{thm}
\begin{proof}
    We sketch a proof based on the proof of \cite[Theorem 9]{OlsZam1976}.
    Let $A$ be a separable, unital, properly infinite \Cs{}.
    Then there exist isometries $s_1,s_2,\ldots\in A$ with pairwise orthogonal ranges (i.e., $A$ contains a unital copy of the Cuntz algebra $\mathcal{O}_\infty$).

    Let $a_1,a_2,\ldots\in A$ be a sequence of (positive) generators for $A$ such that their spectra satisfy $\sigma(a_k)\subset[1/2\cdot 1/4^k,1/4^k]$.
    A generator for $A$ is given by:
    \begin{align*}
        x &:=\sum_{k\geq 1}(s_ka_ks_k^\ast + 1/2^ks_k).
    \end{align*}

    As in in the proof of \cite[Theorem 9]{OlsZam1976}, one can show that $\sigma(x)\subset\{0\}\cup\bigcup_{k\geq 1}[1/2\cdot 1/4^k,1/4^k]$.
    Let $B:=C^*(x)\subset A$.
    Proceeding inductively, one shows that $a_k,s_k\in B$.
    We only sketch this for $k=1$.
    Set $p:=s_1s_1^*$.
    Let $f_n$ be a sequence of polynomials converging uniformly to $1$ on $[1/8,1/4]$ and to $0$ on $[0,1/16]$.
    Then $f_n(x)$ converges to an element $y\in B$ of the form $y=p+pb(1-p)$ for some $b\in A$.
    We compute $yy^*=p(1_A+b(1-p)b^*)p$.
    Then for a continuous function $f\colon\RR\to\RR$ with $f(0)=0$ and $f(t)=1$ for $t\geq 1$, we get $f(yy^*)=p\in B$.
    Then $s_1a_1s_1^*=pxp\in B$ and $s_1=2\cdot px(1-p)\in B$, and then also $a_1\in B$.
\end{proof}

\begin{pargr}
\label{pargr:C_X_alg}
    Let $X$ be a locally compact $\sigma$-compact Hausdorff space.
    A $C_{0}(X)$-algebra is a \Cs{} $A$ together with a morphism $\eta\colon C_0(X)\to Z(M(A))$, from the commutative \Cs{} $C_0(X)$ to the center of the multiplier algebra of $A$, such that for any approximate unit $(u_{\lambda})_{\Lambda}$ of $C_0(X)$, $\eta(u_{\lambda}) a \to a $ for any  $a \in A$, or equivalently, the closure of $\eta(C_0(X)) A$ is all of $A$.
    Thus, if $X$ is compact, then $\eta$ is necessarily unital.
    We will usually suppress reference to the structure map, and simply write $fa$ or $f\cdot a$ instead of $\eta(f)a$ for the product of a function $f\in C_0(X)$ and an element $a\in A$.

    Let $Y\subset X$ be a closed subset, and $U:=X\setminus Y$ its complement (an open subset).
    Then $C_0(U)\cdot A$ is an ideal of $A$, denoted by $A(U)$.
    The quotient $A/A(U)$ is denoted by $A(Y)$.

    Given a point $x\in X$, we write $A(x)$ for $A(\{x\})$, and we call this \Cs{} the fiber of $A$ at $x$.
    For an element $a\in A$, we denote by $a(x)$ the image of $a$ in the fiber $A(x)$.
    For each $a\in A$, we may consider the map $\check{a}\colon x\mapsto\|a(x)\|$.
    This is a real-valued, upper-semicontinuous function on $X$, vanishing at infinity.
    The $C_{0}(X)$-algebra $A$ is called continuous if $\check{a}$ is a continuous function for each $a\in A$.

    For more information on $C_{0}(X)$-algebras we refer the reader to \cite[\S 1]{Kas1988} or the more recent \cite[\S 2]{Dad2009}.
\end{pargr}

\begin{pargr}
\label{pargr:Z}
    The Jiang-Su algebra $\mathcal{Z}$ was constructed in \cite{JiaSu1999}; it may be regarded as a $C^{*}$-algebraic analog of the hyperfinite $\text{II}_1$-factor.
    It can be obtained as an inductive limit of prime dimension drop algebras
    $\mathcal{Z}_{p,q}:=\{f\colon[0,1]\to M_p\otimes M_q\ |\ f(0)\in 1_p\otimes M_q, f(1)\in M_p\otimes 1_q\}$.

    For more details, we refer the reader to \cite{Win2011}, where  $\mathcal{Z}$ is characterized in an entirely abstract manner, and to \cite{Ror2004} and \cite{RorWin2010}, where it is  shown that the generalized dimension drop algebra $\mathcal{Z}_{2^\infty,3^\infty}:=\{f\colon[0,1]\to M_{2^\infty}\otimes M_{3^\infty}\ |\ f(0)\in 1\otimes M_{3^\infty}, f(1)\in M_{2^\infty}\otimes 1\}$ embeds unitally into $\mathcal{Z}$; in fact, $\mathcal{Z}$ can be written as a stationary inductive limit of $\mathcal{Z}_{2^\infty,3^\infty}$.
\end{pargr}

\section{Results}
\label{sect:results}

\begin{lma}
\label{lma:Getting_5}
    Let $A$ be a separable, unital \Cs{}.
    Then $\gen(A\otimes \mathcal{Z}_{2^\infty,3^\infty})\leq 5$.
\end{lma}
\begin{proof}
    Consider the ideal $I:=C_0(0,1)\otimes M_{6^\infty}$ in $B:=A\otimes \mathcal{Z}_{2^\infty,3^\infty}$.
    The quotient $B/I$ is isomorphic to $(A\otimes M_{2^\infty})\oplus (A\otimes M_{3^\infty})$.
    Thus, we have a short exact sequence:
    \begin{center}
        \makebox{
        \xymatrix{
            A\otimes C_0(0,1)\otimes M_{6^\infty} \ar[r]
            & A\otimes \mathcal{Z}_{2^\infty,3^\infty} \ar[r]
            & (A\otimes M_{2^\infty})\oplus (A\otimes M_{3^\infty})
        }}
    \end{center}

    It follows from \cite{OlsZam1976} that the tensor product of a unital, separable \Cs{} with a UHF-algebra is singly generated.
    In particular, $\gen(A\otimes M_{2^\infty}),\gen(A\otimes M_{3^\infty})\leq 2$.
    Thus, the quotient satisfies $\gen(B/I)=\max\{\gen(A\otimes M_{2^\infty}),\gen(A\otimes M_{3^\infty})\}\leq 2$, see \ref{pargr:Generators}.

    Note that $I$ is an ideal in the \Cs{} $C:=A\otimes C(S^1)\otimes M_{2^\infty}$.
    We have $\gen(C)\leq 2$, and then $\gen(I)\leq\gen(C)+1\leq 3$, by \autoref{prop:Gen_ideal}.
    Then, the extension is generated by at most $2+3=5$ self-adjoint elements.
\end{proof}

\noindent
    The following is a Stone-Weierstrass type result.
    We prove it using the factorial Stone-Weierstrass conjecture, which states that a sub-\Cs{} $B\subset A$ exhausts $A$ if it separates the factorial states of $A$.
    The factorial Stone-Weierstrass conjecture was proved for separable \Cs{s} independently by Longo, \cite{Lon1984}, and Popa, \cite{Pop1984}.

    See \ref{pargr:C_X_alg} for a short introduction to $C_{0}(X)$-algebras.

\begin{lma}
\label{lma:fibered_subalg}
    Let $A$ be a separable, continuous $C_{0}(X)$-algebra, and $B\subset A$ a sub-\Cs{} such that the following two conditions are satisfied:
    \begin{enumerate}[(i)]
        \item
        For each $x\in X$, $B$ exhausts the fiber $A(x)$,
        \item
        $B$ separates the points of $X$ by full elements, i.e., for each distinct pair of points $x_0,x_1\in X$ there exists some $b\in B$ such that $b(x_1)$ is full in $B(x_1)=A(x_1)$ and $b(x_0)=0$.
    \end{enumerate}
    Then $A=B$.

    Condition (ii) is for instance satisfied if $B$ contains the image of the structure map $\eta\colon C_0(X)\to Z(M(A))$.
\end{lma}
\begin{proof}
    Set $Y:=\Prim(Z(M(A)))$, and identify $Z(M(A))$ with $C(Y)$.
    Let $\pi\colon A\to B(H)$ be a non-degenerate factor representation.
    Then $\pi$ extends to a representation $\tilde{\pi}\colon M(A)\to B(H)$.
    It is straightforward to show $\pi(A)''=\tilde{\pi}(M(A))''$, so that $\tilde{\pi}$ is a factor representation of $M(A)$.
    For any $c\in Z(M(A))$, we have $c\in \pi(A)'\cap \tilde{\pi}(M(A))'' = \CC\cdot 1_H$.
    Thus, there exists a point $y\in Y$ such that $\tilde{\pi}(c)=c(y)\cdot 1_H$ for all $c\in Z(M(A))$.
    Since $\eta(C_0(X))$ contains an approximate unit for $A$, we have that $\tilde{\pi}\circ\eta$ is non-zero.
    Thus, there exists a point $x\in X$ such that $\tilde{\pi}\circ\eta(f)=f(x)\cdot 1_H$ for all $f\in C_0(X)$.
    This means that $\tilde{\pi}\circ\eta$ vanishes on the ideal $A(X\setminus\{x\})$, so that $\pi$ factors through the fiber $A(x)$.

    Let us show that $B\subset A$ separates the factors states of $A$.
    So let $\varphi_1,\varphi_2$ be two different, non-degenerate factors states of $A$.
    We have shown above that there are two points $x_1,x_2\in X$ such that $\varphi_i$ factors through $A(x_i)$, and we denote by $\bar{\varphi}_i\colon A(x_i)\to\CC$ the induced factor state on $A(x_i)$, for $i=1,2$.
    We distinguish two cases:

    Case 1: $x_1=x_2$.
    In this case, since $\varphi_1\neq\varphi_2$, there exists an element $a\in A$ such that $\varphi_1(a)\neq\varphi_2(a)$.
    By condition (i), there exists some element $b\in B$ such that $b(x_1)=a(x_1)$.
    Note that $\varphi_i(b)=\bar{\varphi}_i(b(x_1))=\bar{\varphi}_i(a(x_1))=\varphi_i(a)$, for $i=1,2$.
    Thus, $b$ separates the two states.

    Case 2: $x_1\neq x_2$.
    In this case, by condition (ii), there exists an element $b\in B$ such that $b(x_2)$ is full in $A(x_2)$ and $b(x_1)=0$.
    Since $\varphi_2\neq 0$, there exists an element $a\in A$ such that $|\varphi_2(a)|=|\bar{\varphi}_2(a(x_2))|\geq 1$.

    Since $b(x_2)$ is full, there exist finitely many elements $g_i,h_i\in A(x_2)$ such that $\|a(x_2)-\sum_i c_ib(x_2)d_i\|<1$.
    By condition (i), there exist elements $\tilde{g}_i,\tilde{h}_i\in B$ such that $\tilde{g}_i(x_2)=g_i$ and $\tilde{h}_i(x_2)=h_i$.
    Set $b':=\sum_i\tilde{c}_ib\tilde{d}_i$.
    Then $|\varphi_2(b')|=|\bar{\varphi}_2(b'(x_2))|>0$, while $b'(x_1)=0$.
    This shows that $b'$ separates the two states.

    It follows that $B$ separates the factor states of $A$, and therefore $B=A$ by the factorial Stone-Weierstrass conjecture, proved independently by Longo, \cite{Lon1984}, and Popa, \cite{Pop1984}.
\end{proof}

\begin{lma}
\label{lma:Reducing_by_1}
    Let $A$ be a unital \Cs{} with $\gen(A)\leq 3$.
    Then there exist a positive element $x\in A\otimes\mathcal{Z}_{2,3}$ and two positive, full elements $y',z'\in \mathcal{Z}_{2,3}$ such that $A\otimes\mathcal{Z}_{2,3}$ is generated by $x$ and $1\otimes y'$, and further $y'$ and $z'$ are orthogonal.
\end{lma}
\begin{proof}
    We consider $\mathcal{Z}_{2,3}$ as the \Cs{} of continuous functions from $[0,1]$ to $M_6$ with the boundary conditions
    \begin{align*}
        f(0)=\begin{pmatrix}
            Y \\
            & Y \\
            & & Y
        \end{pmatrix}
        \quad
        f(1)=\begin{pmatrix}
            Z \\
            & QZQ^* \\
        \end{pmatrix},
    \end{align*}
    where $Y\in M_2$ and $Z\in M_3$ are arbitrary matrices, and $Q\in M_3$ is the following fixed permutation matrix:
    \begin{align*}
        Q=\begin{pmatrix}
            & & 1 \\
            1 \\
            & 1
        \end{pmatrix}.
    \end{align*}

    This means that $f(0),f(1)\in M_6$ have the following form:
    \begin{align*}
        f(0) = \left(\begin{array}{p{9pt}p{9pt}p{9pt}p{9pt}p{9pt}p{9pt}p{0pt}}
            $\mu_{11}$ & $\mu_{12}$ & \\
            $\mu_{21}$ & $\mu_{22}$ & \\
            & & $\mu_{11}$ & $\mu_{12}$ & \\
            & & $\mu_{21}$ & $\mu_{22}$ & \\
            & & & & $\mu_{11}$ & $\mu_{12}$ & \\
            & & & & $\mu_{21}$ & $\mu_{22}$ & \\
        \end{array}\right)
        f(1) = \left(\begin{array}{p{9pt}p{9pt}p{9pt}p{9pt}p{9pt}p{9pt}p{0pt}}
            $\lambda_{11}$ & $\lambda_{12}$ & $\lambda_{13}$ & \\
            $\lambda_{21}$ & $\lambda_{22}$ & $\lambda_{23}$ & \\
            $\lambda_{31}$ & $\lambda_{22}$ & $\lambda_{33}$ & \\
            & & & $\lambda_{33}$ & $\lambda_{31}$ & $\lambda_{32}$ & \\
            & & & $\lambda_{13}$ & $\lambda_{11}$ & $\lambda_{12}$ & \\
            & & & $\lambda_{23}$ & $\lambda_{21}$ & $\lambda_{22}$ & \\
        \end{array}\right),
    \end{align*}
    for numbers $\mu_{i,j},\lambda_{i,j}\in\CC$.

    Note that $\mathcal{Z}_{2,3}$ is naturally a continuous $C([0,1])$-algebra, with fibers $\mathcal{Z}_{2,3}(0)\cong M_2$, $\mathcal{Z}_{2,3}(1)\cong M_3$, and $\mathcal{Z}_{2,3}(t)\cong M_6$ for points $t\in (0,1)\subset[0,1]$.

    Let $a,b,c\in A$ be a set of invertible, positive generators for $A$.
    Denote by $e_{i,j}$ the matrix units in $M_6$.
    To shorten notation, for indices $i,j$ set $f_{i,j}:=e_{i,j}+e_{j,i}$.
    For $t\in[0,1]$ we define the following element of $A\otimes M_6$:
    \begin{align*}
        x_t := & a\otimes (e_{1,1} + (1-t)\cdot e_{3,3} + e_{5,5}) \\
            + & b\otimes (f_{1,2} + (1-t)\cdot f_{3,4} + f_{5,6}) \\
            + & c \otimes (e_{2,2} + (1-t)\cdot e_{4.4} + e_{6,6}) \\
            + & 1_A \otimes (t\cdot f_{2,3} + t\cdot f_{4,5} + \delta(t)\cdot f_{1,3}) \\
    \end{align*}
    where $\delta\colon[0,1]\to[0,1]$ is a continuous function on $[0,1]$ that takes the value $0$ at the endpoints $0$ and $1$, and is strictly positive at each point $t\in (0,1)$, e.g., $\delta$ could be given by $\delta(t)=1/4-(t-1/2)^2$.
    We also define for $t\in[0,1]$ two elements of $M_6$:
    \begin{align*}
        y_t' := & e_{1,1} + (1-t)\cdot e_{3,3} + e_{5,5} \\
        z_t' := & e_{2,2} + (1-t)\cdot e_{4,4} + e_{6,6}
    \end{align*}

    It is easy to check that the assignment $x\colon t\mapsto x_t$ defines an element $x\in A\otimes \mathcal{Z}_{2.3}$.
    Similarly, we get two elements $y',z'\in\mathcal{Z}_{2.3}$ defined via $t\mapsto y_t'$ and $t\mapsto z_t'$.
    In matrix form, these elements look as follows:
    \begin{align*}
        x_t &:=\left( \begin{array}{cc|cc|cc}
            a & b & \delta(t) & & \\
            b & c & t & & \\
            \hline 
            \delta(t) & t & (1-t)a & (1-t)b & \\
            & & (1-t)b & (1-t)c & t \\
            \hline 
            & & & t & a & b \\
            & & & & b & c\\
        \end{array} \right) \\
        y_t' &:=\left( \begin{array}{cc|cc|cc}
            1 & & & \\
            & & & & \\
            \hline 
            & & (1-t) & \\
            & & & & \\
            \hline 
            & & & & 1 \\
            & & & & & \\
        \end{array} \right)
        \quad
        z_t' :=\left( \begin{array}{cc|cc|cc}
            & & & \\
            & 1 & & & \\
            \hline 
            & & & \\
            & & & (1-t) & \\
            \hline 
            & & & & & \\
            & & & & & 1 \\
        \end{array} \right)
    \end{align*}

    Set $y:=1\otimes y'$, and let $D:=C^{*}(x+1,y)$ be the sub-\Cs{} of $E:=A\otimes \mathcal{Z}_{2,3}$ generated by the two self-adjoint elements $x+1$ and $y$.
    Since $x\geq 0$, we get that both $1$ and $x$ lie in $C^{*}(x+1)$.
    It follows that $D=C^{*}(1,x,y)$, and we will show that $D=E$.
    Note that $E$ has a natural continuous $C([0,1])$-algebra structure (induced by the one of $\mathcal{Z}_{2,3}$), with fibers $E(0)\cong A\otimes M_2$, $E(1)\cong A\otimes M_3$, and $E(t)\cong A\otimes M_6$ for points $t\in (0,1)\subset[0,1]$.

    Let $J:=E((0,1))\lhd E$ be the natural ideal corresponding to the open set $(0,1)\subset[0,1]$.
    Note that $J\cong A\otimes C_0((0,1))\otimes M_6$, and $J$ is naturally a continuous $C_{0}((0,1))$-algebra.
    We will show in two steps that $D$ exhausts the ideal $J$ (i.e., $D\cap J=J$) and the quotient $E/J$ (i.e., $D/(D\cap J)=E/J$).

    Step 1:
    We want to apply \autoref{lma:fibered_subalg} to the $C((0,1))$-algebra $J$ with sub-\Cs{} $D\cap J$.
    To verify condition (ii), note that the \Cs{} generated by $y'$ contains $C_0((0,1))\otimes e_{3,3}$.
    Therefore, $D\cap J$ contains $1_A\otimes C_0((0,1))\otimes e_{3,3}$, which separates the points of $(0,1)$.
    Since $1_A\otimes e_{3,3}\in E(t)\cong A\otimes M_6$ is full,  condition (ii) of  \autoref{lma:fibered_subalg} holds and it remains to verify condition (i).

    We need to show that $D\cap J$ exhausts all fibers of $J$.
    Fix some $t\in (0,1)$, and set $D_t:=C^*(1,x_t,y_t)\subset A\otimes M_6$.
    To simplify notation, we write $\bar{e}_{i,j}$ for the matrix units $1_A\otimes e_{i,j}\in A\otimes M_6$.
    We need to show that $D_t$ is all of $A\otimes M_6$.
    This will follow if $D_t$ contains all $\bar{e}_{i,j}$, and for this it is enough to show that the off-diagonal matrix units $\bar{e}_{i,i+1}$ are in $D_t$, for $i=1,\ldots,5$.

    The spectrum of $y_t$ is $\{0,1-t,1\}$.
    Applying functional calculus to $y_t$ we obtain that the following three elements lie in $D_t$:
    \begin{align*}
        u &:=\bar{e}_{1,1}+\bar{e}_{5,5} \\
        v &:=\bar{e}_{3,3} \\
        w &:=1-v-u = \bar{e}_{2,2}+\bar{e}_{4,4}+\bar{e}_{6,6}
    \end{align*}

    Then, we proceed as follows:
    \begin{enumerate}[1. ]
        \item
        $\bar{e}_{1,3} = \delta(t)^{-1}ux_tv\in D_t$ and so $\bar{e}_{1,1},\bar{e}_{5,5}\in D_t$.
        \item
        $g:=b\otimes e_{1,2} = \bar{e}_{1,1}x_tw\in D_t$.
         It follows $b\otimes e_{1,1}=(gg^*)^{1/2}\in D_t$, cf.\  \cite{OlsZam1976}.
        Then $b^{-1}\otimes e_{1,1}\in C^{*}(b\otimes e_{1,1})\subset D_t$ and so $\bar{e}_{1,2}=(b^{-1}\otimes e_{1,1})\cdot g\in D_t$
        and $\bar{e}_{2,2}\in D_t$.
        \item
        $b\otimes e_{3,4} = (1-t)^{-1}\bar{e}_{3,3}x_t(w-\bar{e}_{2,2})\in D_t$.
        Arguing as above, it follows that $\bar{e}_{3,4}\in D_t$, and then $\bar{e}_{4,4},\bar{e}_{6,6}\in D_t$.
        \item
        $\bar{e}_{2,3} = t^{-1}\bar{e}_{2,2}x_t\bar{e}_{3,3}\in D_t$.
        \item
        $\bar{e}_{4,5} = t^{-1}\bar{e}_{4,4}x_t\bar{e}_{5,5}\in D_t$.
        \item
        $b\otimes e_{5,6} = \bar{e}_{5,5}x_t\bar{e}_{6,6}\in D_t$ and so $\bar{e}_{5,6}\in D_t$.
    \end{enumerate}

    This shows that $D\cap J$ exhausts the fibers of $J$.
    We may apply \autoref{lma:fibered_subalg} and deduce $D\cap J=J$, which finishes step 1.

    Step 2:
    We want to show that $D/J$ exhausts $E/J=E(\{0,1\})\cong A\otimes (M_2 \oplus M_3)$.
    Let us denote the matrix units in $M_2$ by $e_{i,j}^{(0)}$, $i=1,2$, and the matrix units in $M_3$ by $e_{i,j}^{(1)}$, $i=1,2,3$.
    To simplify notation, we write $\bar{e}_{i,j}^{(k)}$ for the matrix units $1_A\otimes e_{i,j}^{(k)}\in A\otimes (M_2 \oplus M_3)$.
    Let us denote the image of $x$ and $y$ in $D/J$ by $v$ and $w$:
    \begin{align*}
        v &= a\otimes (e_{1,1}^{(0)} + e_{1,1}^{(1)})
        + b\otimes (e_{1,2}^{(0)} + e_{2,1}^{(0)} + e_{1,2}^{(1)} + e_{2,1}^{(1)} )
        + c\otimes (e_{2,2}^{(0)}+ e_{2,2}^{(1)})
        + \bar{e}_{2,3}^{(1)} + \bar{e}_{3,2}^{(1)}\\
        &= \begin{pmatrix}
            a & b\\
            b & c \\
        \end{pmatrix}
        \oplus
        \begin{pmatrix}
            a & b\\
            b & c & 1 \\
            & 1 & \\
        \end{pmatrix}
        \\
        w &= \bar{e}_{1,1}^{(0)} + \bar{e}_{1,1}^{(1)}
        =\begin{pmatrix}
            1 & 0 \\
            0 & 0 \\
        \end{pmatrix}
        \oplus
        \begin{pmatrix}
            1 & & \\
            & 0 & \\
            & & 0 \\
        \end{pmatrix}.
    \end{align*}
    As in step 1, it is enough to show that $D/J$ contains the off-diagonal matrix units $\bar{e}_{1,2}^{(0)}$, $\bar{e}_{1,2}^{(1)}$ and $\bar{e}_{2,3}^{(1)}$.
    We argue as follows:
    \begin{enumerate}[1.  ]
        \item
        $g:=wv(1-w)=b\otimes (e_{1,2}^{(0)} + e_{1,2}^{(1)})\in D/J$.
        As in step 1, it follows that $b\otimes (e_{1,1}^{(0)} + e_{1,1}^{(1)})=(gg^*)^{1/2}\in D/J$.
        Then $b^{-1}\otimes (e_{1,1}^{(0)} + e_{1,1}^{(1)})\in D/J$, and so $\bar{e}_{1,2}^{(0)} + \bar{e}_{1,2}^{(1)}=(b^{-1}\otimes (e_{1,1}^{(0)} + e_{1,1}^{(1)}))\cdot g\in D/J$.
        It follows that $\bar{e}_{2,2}^{(0)} + \bar{e}_{2,2}^{(1)}\in D/J$.
        \item
        $\bar{e}_{3,3}^{(1)}=1-w-(\bar{e}_{2,2}^{(0)} + \bar{e}_{2,2}^{(1)})\in D/J$.
        \item
        $\bar{e}_{2,3}^{(1)}= v \bar{e}_{3,3}^{(1)}\in D/J$, and so $\bar{e}_{2,2}^{(1)}\in D/J$.
        \item
        $b\otimes e_{1,2}^{(1)}=wv\bar{e}_{2,2}^{(1)}\in D/J$.
        Again, this implies $\bar{e}_{1,2}^{(1)}\in D/J$ and so $\bar{e}_{1,1}^{(1)}\in D/J$.
        \item
        $\bar{e}_{1,1}^{(0)}=w-\bar{e}_{1,1}^{(1)}\in D/J$.
        \item
        $\bar{e}_{2,2}^{(0)}=1-w-\bar{e}_{2,2}^{(1)}-\bar{e}_{3,3}^{(1)}\in D/J$.
        \item
        $b\otimes e_{1,2}^{(0)}=\bar{e}_{1,1}^{(0)}v\bar{e}_{2,2}^{(0)}\in D/J$.
        Again, this implies $\bar{e}_{1,2}^{(0)}\in D/J$.
    \end{enumerate}
    This finishes step 2.

    We have seen that $A\otimes\mathcal{Z}_{2,3}$ is generated by $x+1$ and $y$.
    Moreover, $z'$ is full, positive and orthogonal to $y'$.
\end{proof}

\begin{lma}
\label{lma:Tensor_inf_dd_2}
    Let $A$ be a separable, unital \Cs{}.
    Then there exist a positive element $x\in A\otimes \mathcal{Z}_{2^\infty,3^\infty}$ and two positive, full elements $y',z'\in \mathcal{Z}_{2^\infty,3^\infty}$ such that $A\otimes \mathcal{Z}_{2^\infty,3^\infty}$ is generated by $x$ and $y:=1\otimes y'$, and further $y'$ and $z'$ are orthogonal.
\end{lma}
\begin{proof}
    Let $B:=A\otimes\mathcal{Z}_{2^\infty,3^\infty}$.
    Note that $\mathcal{Z}_{2^\infty,3^\infty}\otimes\mathcal{Z}_{2,3}$ is naturally a $C([0,1]\times[0,1]$)-algebra.
    Then, the quotient corresponding to the diagonal $\{(t,t)\setDefSep t\in[0,1]\}\subset[0,1]\times[0,1]$ is isomorphic to $\mathcal{Z}_{2^\infty,3^\infty}$, and we denote the resulting surjective morphism by $\pi\colon\mathcal{Z}_{2^\infty,3^\infty}\otimes\mathcal{Z}_{2,3}\to\mathcal{Z}_{2^\infty,3^\infty}$.
    We proceed in two steps.

    Step $1$:
    We show that $\gen(B)\leq k+1$ implies $\gen(B)\leq k$ for $k\geq 2$.
    So assume $B$ is generated by the self-adjoint, invertible elements $a_1,\ldots,a_{k+1}$.
    The sub-\Cs{} $C:=C^{*}(a_{k-1},a_k,a_{k+1})\subset B$ is unital and satisfies $\gen(C)\leq 3$.
    Consider the \Cs{} $B\otimes\mathcal{Z}_{2,3}$.
    By \autoref{lma:Reducing_by_1}, the sub-\Cs{} $C\otimes\mathcal{Z}_{2,3}$ is generated by two self-adjoint elements, say $b,c$.

    One readily checks that $B\otimes\mathcal{Z}_{2,3}$ is generated by the $k$ self-adjoint elements $a_1\otimes 1,\ldots,a_{k-2}\otimes 1,b,c$.
    Since $B=A\otimes\mathcal{Z}_{2^\infty,3^\infty}$ is isomorphic to a quotient of $B\otimes\mathcal{Z}_{2,3}=A\otimes\mathcal{Z}_{2^\infty,3^\infty}\otimes\mathcal{Z}_{2,3}$, we obtain $\gen(B)\leq\gen(B\otimes\mathcal{Z}_{2,3})\leq k$.

    Step $2$:
    By \autoref{lma:Getting_5}, we have $\gen(B)\leq 5$.
    Applying Step $1$ several times, we obtain $\gen(B)\leq 3$.

    It follows from \autoref{lma:Reducing_by_1} that there exists a positive element $\tilde{x}\in B\otimes\mathcal{Z}_{2,3}$ and two positive, full elements $\tilde{y}',\tilde{z}'\in \mathcal{Z}_{2,3}$ such that $B\otimes\mathcal{Z}_{2,3}$ is generated by $\tilde{x}$ and $1\otimes \tilde{y}'$, and further $\tilde{y}'$ and $\tilde{z}'$ are orthogonal.

    Consider the surjective morphism $\id\otimes\pi\colon A\otimes\mathcal{Z}_{2^\infty,3^\infty}\otimes\mathcal{Z}_{2,3} \to A\otimes\mathcal{Z}_{2^\infty,3^\infty}$.
    One checks that the elements $x:=(\id\otimes\pi)(\tilde{x})\in A\otimes\mathcal{Z}_{2^\infty,3^\infty}$, and $y':=\pi(\tilde{y}'), z':=\pi(\tilde{z}')\in\mathcal{Z}_{2^\infty,3^\infty}$ have the desired properties.
\end{proof}

\begin{thm}
\label{thm:Tensor_gen2}
    Let $A, B$ be two separable, unital \Cs{s}.
    Assume the following:
    \begin{enumerate}
        \item
        $A$ contains a sequence $a_1,a_2,\ldots$ of full, positive elements that are pairwise orthogonal,
        \item
        $B$ admits a unital embedding of the Jiang-Su algebra $\mathcal{Z}$.
    \end{enumerate}
    Then $A\otimes_\mx B$ is singly generated.
    Every other tensor product $A\otimes_\lambda B$ is a quotient of $A\otimes_\mx B$, and therefore is also singly generated.
\end{thm}
\begin{proof}
    There exists a unital embedding of $\mathcal{Z}_{2^\infty,3^\infty}$ in $\mathcal{Z}$, so we may assume that there is a unital embedding of $\mathcal{Z}_{2^\infty,3^\infty}$ in $B$.
    We may assume that the elements $a_1,a_2,\ldots\in A$ are contractive.

    Choose a sequence $b_1,b_2,\ldots\in B$ of contractive, positive elements that is dense in the set of all contractive, positive elements of $B$.

    Consider the sub-\Cs{} $A\otimes \mathcal{Z}_{2^\infty,3^\infty}\subset A\otimes_\mx B$.
    By \autoref{lma:Tensor_inf_dd_2}, there exist a positive element $x\in A\otimes \mathcal{Z}_{2^\infty,3^\infty}$ and two full, positive elements $y',z'\in \mathcal{Z}_{2^\infty,3^\infty}$ such that $A\otimes \mathcal{Z}_{2^\infty,3^\infty}$ is generated by $x$ and $y:=1\otimes y'$, and further $y'$ and $z'$ are orthogonal.

    Define the following two elements of $A\otimes_\mx B$:
    \begin{align*}
        v := x, \quad\quad\quad w:= 1\otimes y' - \sum_{k\geq 1} 1/2^k\cdot a_k\otimes (z'b_kz').
    \end{align*}

    Let $D:=C^*(v,w)$ be the sub-\Cs{} of $A\otimes_\mx B$ generated by $v$ and $w$.
    We claim that $D=A\otimes B$.

    Step $1$: We show $A\otimes \mathcal{Z}_{2^\infty,3^\infty}\subset D$.
    Note that the two elements $1\otimes y'$ and $\sum_{k\geq 1} 1/2^k\cdot a_k\otimes (z'b_kz')$ are positive and orthogonal.
    It follows that $1\otimes y'$ is the positive part of $w$, and therefore $1\otimes y'\in D$.
    Therefore, $C^*(v,1\otimes y')=A\otimes \mathcal{Z}_{2^\infty,3^\infty}\subset D$.

    Step $2$: We show $1\otimes B\subset D$.
    We have $g:=\sum_{k\geq 1} 1/2^k\cdot a_k\otimes (z'b_kz')\in D$.
    It follows from Step $1$ that $a_k\otimes 1\in D$, and so $a_k^2\otimes (z'b_kz')=2^k\cdot(a_k\otimes 1)g\in D$.
    Since $a_k^2$ is full, there exist finitely many elements $c_i,d_i\in A$ such that $1_A=\sum_ic_ia_k^2d_i$.
    By Step $1$, we have $c_i\otimes 1,d_i\otimes 1\in D$.
    Then $1\otimes (z'b_kz') = \sum_i(c_i\otimes 1) (a_k^2\otimes (z'b_kz')) (d_i\otimes 1)\in D$, for each $k$.

    Let $b\in B$ be a contractive, positive element.
    Then $b=\lim_j b_{k(j)}$ for certain indices $k(j)$.
    Then $1\otimes (z'bz') = \lim_j 1\otimes (z'b_{k(j)}z')\in D$.
    It follows that the hereditary sub-\Cs{} $1\otimes \overline{z'Bz'}$ is contained in $D$.
    Since $z'$ is full in $\mathcal{Z}_{2^\infty,3^\infty}$, there exist finitely many elements $c_i,d_i\in \mathcal{Z}_{2^\infty,3^\infty}$ such that $1_B=\sum_ic_iz'd_i$.
    We have seen that $1\otimes z'bz'\in D$ for any $b\in B$.
    Then $1\otimes bz' = \sum_i (1\otimes c_i)(1\otimes z'd_ibz')\in D$ for any $b\in B$.
    Similarly $1\otimes b = \sum_i (1\otimes bc_iz')(1\otimes d_i)\in D$ for any $b\in B$, as desired.

    It follows from Steps $1$ and $2$ that for each $a\in A$ and $b\in B$ the simple tensor $a\otimes b$ is contained in $D$.
    The conclusion follows since $A\otimes_\mx B$ is the closure of the linear span of simple tensors.
\end{proof}

\begin{cor}
\label{prop:Gen2_both_Z_embed}
    Let $A, B$ be two separable, unital \Cs{s} that both admit a unital embedding of the Jiang-Su algebra $\mathcal{Z}$.
    Then $A\otimes_\mx B$ is singly generated.
\end{cor}
\begin{proof}
    It is easy to verify that condition $(i)$ of \autoref{thm:Tensor_gen2} is fulfilled if $A$ admits a unital embedding of $\mathcal{Z}$.
\end{proof}

\begin{thm}
\label{prop:Gen2_Z-stable}
    Let $A$ be a unital, separable \Cs{}.
    Then $A\otimes\mathcal{Z}$ is singly generated.
\end{thm}
\begin{proof}
    Note that $A\otimes\mathcal{Z}\cong (A\otimes\mathcal{Z})\otimes\mathcal{Z}$.
    It is clear that both $A\otimes\mathcal{Z}$ and $\mathcal{Z}$ admit unital embeddings of $\mathcal{Z}$.
    Then apply the above \autoref{prop:Gen2_both_Z_embed}.
\end{proof}

\begin{cor}
\label{prop:Gen3_non-unital_Z-stable}
    Let $A$ be a separable \Cs{}.
    Then $\gen(A\otimes \mathcal{Z})\leq 3$.
\end{cor}
\begin{proof}
    Let $\widetilde{A}$ be the minimal unitization of $A$.
    It follows from \autoref{prop:Gen2_Z-stable} that $\gen(\widetilde{A}\otimes\mathcal{Z})\leq 2$.
    Since $A\otimes\mathcal{Z}$ is an ideal in $\widetilde{A}\otimes\mathcal{Z}$, we get $\gen(A\otimes\mathcal{Z})\leq\gen(\widetilde{A}\otimes\mathcal{Z})+1\leq 3$ from \autoref{prop:Gen_ideal}, as desired.
\end{proof}

\noindent
    Our results allow us to give new proofs for results about single generation of certain  von Neumann algebras.

\begin{prop}
\label{prop:tensor_vNalg}
    Assume $M,N$ are separably-acting von Neumann algebras that both admit a unital embedding of the hyperfinite $\text{II}_1$-factor.
    Then $M\bar{\otimes}N$ is singly generated.
\end{prop}
\begin{proof}
    Consider the GNS-representation $\pi\colon\mathcal{Z}\to B(H)$ of the Jiang-Su algebra with respect to its tracial state.
    The weak closure, $\pi(\mathcal{Z})''$, is isomorphic to the hyperfinite $\text{II}_1$-factor $\mathcal{R}$.
    Thus, there exists a weakly dense, unital copy of $\mathcal{Z}$ inside $\mathcal{R}$.

    Choose  weakly dense, separable, unital \Cs{s} $A_0\subset M$, and similarly $B_0\subset N$.
    Consider $\mathcal{Z}\subset\mathcal{R}\subset M$ and set $A:=C^*(A_0,\mathcal{Z})\subset M$.
    Similarly set $B:=C^*(B_0,\mathcal{Z})\subset N$.

    Then $A$ and $B$ are separable, unital \Cs{s} that both contain unital copies of the Jiang-Su algebra.
    By \autoref{prop:Gen2_both_Z_embed}, $A\otimes_\mx B$ is singly generated.

    Consider the sub-\Cs{} $C:=C^*(A\bar{\otimes}1,1\bar{\otimes}B)\subset M\bar{\otimes}N$.
    Then $C$ is a quotient of $A\otimes_\mx B$, and therefore singly generated.
    Since $C$ is weakly dense in $M\bar{\otimes}N$, we obtain that $M\bar{\otimes}N$ is singly generated, as desired.
\end{proof}

\begin{rmk}
    We note that a von Neumann algebra $M$ admits a unital embedding of $\mathcal{R}$ if and only if $M$ has no (non-zero) finite-dimensional representations.

    The analogous statement for \Cs{s} would be that a \Cs{} $A$ admits a unital embedding of $\mathcal{Z}$ if and only if $A$ has no (non-zero) finite-dimensional representations.
    It was shown by Elliott and R{\o}rdam, \cite{EllRor2006}, that this is true for \Cs{s} of real rank zero.
    However, in \cite{DadHirTomWin2009} a simple, separable, unital, non-elementary AH-algebra is constructed into which $\mathcal{Z}$ does not embed.
\end{rmk}

\noindent
    As a particular case of \autoref{prop:tensor_vNalg} we obtain the following result of Ge and Popa.

\begin{cor}[{Ge, Popa, \cite[Theorem 6.2]{GePop1998}}]
\label{prop:GePope}
    Assume $M,N$ are separably-acting $\text{II}_1$-factors.
    Then $M\bar{\otimes}N$ is singly generated.
\end{cor}

\section{Applications}
\label{sect:appl}

\noindent
    In this section we show that the Jiang-Su algebra $\mathcal{Z}$ embeds unitally into the reduced group \Cs{s}, $C^*_r(\Gamma)$, of groups $\Gamma$ that contain a non-cyclic free subgroup, see \autoref{prop:Z-embedding_gps}.
    We only consider discrete groups, and we let $F_k$ denote the free group with $k$ generators ($k\in\{2,3\ldots,\infty\}$).

    We can apply \autoref{thm:Tensor_gen2} to show that certain tensor products of the form $A\otimes_\mx C^*_r(\Gamma)$ are singly generated, see \autoref{prop:Gen2_tensor_free_gp_alg}.
    In particular, $C^*_r(F_\infty)\otimes C^*_r(F_\infty)$ is singly generated, although it is not $\mathcal{Z}$-stable, see \autoref{pargr:free_gp_tensor_free_gp}.

\begin{pargr}
\label{pargr:Z-embedding_free_gp_alg}
    It was shown by Robert, \cite{Rob2010}, that the Jiang-Su algebra $\mathcal{Z}$ embeds unitally into $C^*_r(F_\infty)$.
    A key observation is that $C^*_r(F_\infty)$ has strict comparison of positive elements.
    This follows from the work of Dykema and R{\o}rdam on reduced free product \Cs{s}, see \cite{DykRor1998b} and \cite{DykRor2000}.

    Dykema and R{\o}rdam study the comparison of projections, but this can be generalized to obtain results about the comparison of positive elements, as noted by Robert, \cite{Rob2010}.
    In particular, \cite[Lemma 5.3]{DykRor1998b} and \cite[Theorem 2.1]{DykRor2000} can be generalized, and it follows that $C^*_r(F_\infty)$ has strict comparison of positive elements.
\end{pargr}

\begin{prop}
\label{prop:Z-embedding_gps}
       If $\Gamma$ is a discrete group that contains $F_\infty$ as a subgroup, then $\mathcal{Z}$ embeds unitally into $C^*_r(\Gamma)$.
\end{prop}
\begin{proof}
    In general, for any subgroup $\Gamma_1$ of a discrete group $\Gamma$, we have a unital embedding $C^*_r(\Gamma_1)\subset C^*_r(\Gamma)$.
    Hence, if $F_\infty$ is a subgroup of $\Gamma$, then $C^*_r(\Gamma)$ contains a unital copy of $C^*_r(F_\infty)$, which in turn contains a unital copy of $\mathcal{Z}$.
\end{proof}

\begin{rmk}
\label{pargr:gps_with_free_subgp}
    Every non-cyclic free group $F_k$ ($k\geq 2$) contains $F_\infty$ as a subgroup.
    In general, by the Nielsen-Schreier theorem, every subgroup of a free group is again free.
    Thus, if $a,b$ are free elements, then the the elements $a^kb^k$ generate a subgroup $\Gamma=\langle a^kb^k, k\geq 1\rangle$ that is free, and since none of the elements $a^kb^k$ is contained in the subgroup generated by the other elements, we have $\Gamma\cong F_\infty$.

    Thus, when we ask which discrete groups contain $F_\infty$ as a subgroup, we are equivalently asking which groups $\Gamma$ contain a non-cyclic free subgroup.
    It is a necessary condition that $\Gamma$ is non-amenable.
    The converse implication is known as the von Neumann conjecture, but this  was disproved in 1980 by Ol'shanskij.

    A counterexample are the so-called Tarski monster groups, in which every non-trivial proper subgroup is cyclic of some fixed prime order.
    Clearly, such a group cannot contain $F_\infty$ as a subgroup, and it is Ol'shanskij's contribution to show that Tarski monster groups exist and are non-amenable.

    On the other hand, every group with the weak Powers property, as defined in \cite{BocNic1988}, has a non-cyclic free subgroup.
    A proof can be found in \cite{dlH2007}, which also lists classes of groups that have the (weak) Powers property.
    We just mention that all free products $\Gamma_1\ast \Gamma_2$ with $|\Gamma_1|\geq 2, |\Gamma_2|\geq 3$ have the Powers property, and therefore \autoref{prop:Z-embedding_gps} applies.

    We may derive the following from \autoref{thm:Tensor_gen2} and \autoref{prop:Z-embedding_gps}:
\end{rmk}

\begin{cor}
\label{prop:Gen2_tensor_free_gp_alg}
    Let $A$ be a separable, unital \Cs{} that contains a countable sequence of pairwise orthogonal, full elements (e.g., $A$ is simple and nonelementary), and let $\Gamma$ be a group that contains a non-cyclic free subgroup.
    Then $A\otimes_\mx C^*_r(\Gamma)$ is singly generated.
\end{cor}

\begin{exmpl}
\label{pargr:free_gp_tensor_free_gp}
    Let $\Gamma_1,\Gamma_2$ be two groups that contain non-cyclic free subgroups.
    Then $C^*_r(\Gamma_1\times \Gamma_2)\cong C^*_r(\Gamma_1)\otimes_\mx C^*_r(\Gamma_2)$ is singly generated.
    For example, for any $k,l\in\{2,3,\ldots,\infty\}$, the \Cs{} $C^*_r(F_k)\otimes_{\mx} C^*_r(F_l)$ is singly generated.
    In particular, $C^*_r(F_\infty) \otimes_{\mx} C^*_r(F_\infty)$ is singly generated.

    It was pointed out to the authors by S.\ Wassermann that $C^*_r(F_k)\otimes C^*_r(F_l)$ is not $\mathcal{Z}$-stable, for any $k,l\in\{2,3,\ldots,\infty\}$.
    In fact, if $C^*_r(F_k)\otimes C^*_r(F_l)\cong A\otimes B\otimes C$, then one of the three algebras $A,B$ or $C$ is isomorphic to $\CC$.
    This is a generalization of the fact that $C^*_r(F_k)$ is tensorially prime, and it can be proved similarly.

    We note that it is a difficult open problem whether $C^*_r(F_k)$ is singly generated itself.
\end{exmpl}

\begin{quest}
\label{quest:free_subgp_non-amen}
    Given a non-amenable (discrete) group $\Gamma$.
    Does $C^*_r(\Gamma)$ admit a unital embedding of $\mathcal{Z}$?
\end{quest}

\noindent
    For each group $\Gamma$, the trivial group-morphism $\Gamma\to\{1\}$ induces a surjective morphism $C^*(\Gamma)\to\CC$.
    Thus, the Jiang-Su algebra can never unitally embed into a full group \Cs{}.
    If $\Gamma$ is amenable, then $C^*_r(\Gamma)\cong C^*(\Gamma)$, and consequently there is no unital embedding of $\mathcal{Z}$ into the reduced group \Cs{} of an amenable group.

    On the other hand, if $\Gamma$ contains a non-cyclic free subgroup, then \autoref{prop:Z-embedding_gps} gives a positive answer to \autoref{quest:free_subgp_non-amen}.
    As noted in \autoref{pargr:gps_with_free_subgp}, not every non-amenable group contains a non-cyclic free subgroup.
    However, it is known that the reduced group \Cs{} of a non-amenable group has no finite-dimensional representations, which is a necessary condition for the Jiang-Su algebra to embed.

\section*{Acknowledgments}

\noindent
The first named author thanks Mikael R{\o}rdam for valuable comments, especially on the applications in \autoref{sect:appl}.\\


\providecommand{\bysame}{\leavevmode\hbox to3em{\hrulefill}\thinspace}
\providecommand{\MR}{\relax\ifhmode\unskip\space\fi MR }
\providecommand{\MRhref}[2]{%
  \href{http://www.ams.org/mathscinet-getitem?mr=#1}{#2}
}
\providecommand{\href}[2]{#2}

\end{document}